\documentclass[11pt,twoside]{amsart}
\usepackage{verbatim, amscd, epsfig,amsmath,amsthm,amstext,amssymb,hyperref,tikz
}
\usepackage[all]{xy}
\usepackage{mathrsfs}
\usetikzlibrary{decorations.markings}

\theoremstyle{plain}
\newtheorem*{theorem*}{Theorem}
\newtheorem{theorem}{Theorem}[section]
\newtheorem{lemma}[theorem]{Lemma}

\newtheorem{corollary}[theorem]{Corollary}

\theoremstyle{definition}

\newtheorem {proposition}{Proposition}

\renewcommand{\Re}{{\rm Re}\,}
\renewcommand{\Im}{{\rm Im}\,}

\newcommand{\R}{\mathbb{R}}
\newcommand{\C}{\mathbb{C}}
\newcommand{\Z}{\mathbb{Z}}

\renewcommand{\P}{\mathbb{P}}

\newcommand{\CP}{\C\P}

\newcommand{\Grass}{\mbox{\sffamily Gr}}
\newcommand{\Rm}{\mathcal{R}}
\newcommand{\Hm}{\mathcal{H}}
\newcommand{\Pm}{\mathcal{P}}
\newcommand{\Sm}{\mathcal{S}}
\newcommand{\Frame}{\mathcal{F}}
\newcommand{\B}{\mathcal{B}}
\newcommand{\wn}{n}

\DeclareMathOperator{\Gr}{Gr}
\DeclareMathOperator{\greatest}{gcd}
\DeclareMathOperator{\dimension}{dim}

\DeclareMathOperator{\sign}{sign}
\DeclareMathOperator{\Jac}{Jac}

\begin{document}
\title
[Tori amongst CMC planes in $\R^3$]
{The Prevalence of Tori amongst Constant Mean Curvature Planes in $ \R^3 $ }

\author[E. Carberry]{Emma Carberry}
\email{emma.carberry@sydney.edu.au}
\address{School of Mathematics and Statistics\\University of Sydney\\Australia}
\author[M. Schmidt]{Martin Ulrich Schmidt}
\email{schmidt@math.uni-mannheim.de}
\address{Mathematics Chair III\\
Universit\"at Mannheim\\
D-68131 Mannheim, Germany}
\begin{abstract}
 Constant mean curvature (CMC) tori in Euclidean 3-space are described by an algebraic curve, called the spectral curve, together with a line bundle on this curve and a point on $ S ^ 1 $, called the Sym point. For a given spectral curve the possible choices of line bundle and Sym point are easily described. The space of spectral curves of tori  is totally disconnected. Hence to characterise the ``moduli space" of CMC tori one should, for each genus $ g $, determine the closure $\overline {\Pm^ g} $ of spectral curves of CMC tori within the spectral curves of CMC planes having spectral genus $g$. We identify a real subvariety $\Rm^g $ and a subset $\Sm^ g\subseteq \Rm^g $ such that $\Rm^ g_{\text {max }}\subseteq\overline {\Pm^ g}\subseteq\Sm^ g $, where $\Rm^ g_{\text {max }} $ denotes the points of $\Rm^g$ having maximal dimension.  The lowest spectral genus for which tori exist is $g=2$ and in this case $\Rm^2=\Rm^2_{\text{max}}=\overline{\Pm^2}=\Sm^2$. For $ g >2 $, we conjecture that $\Rm^g\supsetneq\Rm^g_{\text{max}}=\Sm^g$. We give a number of alternative characterisations of $\Rm^ g_{\text {max}} $ and in particular introduce a new integer invariant of a CMC plane of finite type, called its winding number.
\end{abstract}

\date{\today}

\maketitle

\section{Introduction}
 Amongst constant mean curvature (CMC) immersions of the plane into $\R ^ 3 $, those of {\it finite-type} play a special role. They may be described in purely algebraic-geometric terms; there is \cite {PS:89, Hitchin:90, Bobenko:91} an explicit bijection between CMC planes of finite type and spectral data, consisting of a Sym point $\lambda_0\in S^1$, a hyperelliptic curve $ X $ with an anti-holomorphic involution $\rho $, a pair of marked points $x_0, x_\infty $ which are exchanged by $\rho $ and a line bundle on $ X $ of degree $ g +1 $, quaternionic with respect to $\rho $. Here $ g $ denotes the arithmetic genus of the spectral curve $ X $. To specify an immersion we must first provide the data of the spectral curve and a choice of Sym point and then the corresponding isospectral set of CMC immersions is parameterised by a real slice  of the Jacobian $\Jac(X) $ of $ X $.

A fundamental result is that all CMC immersions of genus one surfaces into $\R ^ 3 $ are of finite type \cite {PS:89, Hitchin:90}. The double-periodicity of the immersion is encoded in periodicity conditions for the spectral curve $ X $, with the line bundle again being freely chosen from a real subspace of $\Jac(X) $. These periodicity conditions are transcendental and it is not obvious even that tori of arbitrary spectral genus exist. In fact $ g=0,1$ can be easily excluded but for higher spectral genus, existence was established in  \cite {EKT:93, Jaggy:94}.

The space of spectral curves of CMC tori in Euclidean 3-space is totally disconnected.
Hence we cannot directly construct the ``moduli space" of CMC tori if we wish this space to have some reasonable structure, such as that of a complex analytic space. Instead the natural goal is to determine the closure of spectral curves of CMC tori within the space $ \Hm^ g $ of spectral curves of CMC immersions of finite type. We identify a subset $ \Sm^ g $ of the space of spectral curves of finite type which contains this closure and give a number of equivalent sufficient conditions for a point of $\Sm^ g $ to lie in the closure. We conjecture that these conditions are indeed satisfied by all points of $\Sm^ g $. Proving this appears to require the development of a stronger deformation theory and we hope to return to this in future work.

In stark contrast to the Euclidean situation, spectral curves of CMC tori in the 3-sphere admit a real one-parameter space of deformations, corresponding to changing the mean curvature of the immersion. In \cite {CS:12}, we showed that in the spherical case the closure of the set of spectral curves of CMC tori is in fact the entire set of spectral curves of CMC planes of finite type. The Euclidean case is both more difficult and more interesting.
%

We begin, in section 2, by summarising the spectral curve correspondence and introducing the aforementioned subsets $\Pm ^ g$, $\Sm ^ g$, $\Rm ^ g $ of the space $\Hm^ g $ of CMC planes of finite type, which play a central role in this paper. In section~\ref{section:decomposition} we associate to any CMC plane of finite type a meromorphic function $ f:\P ^ 1\rightarrow \P ^ 1 $, defined up to M\"obius transformations, whose degree and winding number prove to be useful invariants for understanding the moduli space $\Hm ^ g $ and the tori within it. Theorem~\ref {lemma:equivalent} shows that for positive spectral genus, the degree of $ f $ is strictly larger both than one and than the winding number. Theorem~\ref{theorem:decomposition} uses the winding number to give a decomposition of $\Hm ^ g\setminus\Sm ^ g $ into finitely many open, nonempty and disjoint subsets. Whitham deformations and the role they play in this paper are explained in section~\ref{sec:whitham}. Section~\ref{sec:closure} contains the statement and proof of our main result, Theorem~\ref{theorem:main}. We conclude in section~\ref{sec:grassmannian} by analysing subsets of a Grassmannian analogous to, but simpler than, the spaces $\Pm^g$, $\Sm^g$ and $\Rm^g$.

\section{Spectral curves of constant mean curvature tori in $ \R^3 $ }
We shall make use of the description of constant mean curvature (CMC) immersions of genus one surfaces in $ \R^3 $ in terms of spectral curve data \cite{Hitchin:90, PS:89, Bobenko:91}. Each such immersion corresponds to a quintuple $ (X,\lambda,\rho, \lambda_0, L) $ where $ X $ is an algebraic curve, called the {\it spectral curve}, with a degree two meromorphic function $ \lambda $,  anti-holomorphic involution $ \rho $, and a line bundle $ L $ on this curve, which is quaternionic with respect to $ \sigma\rho $, where $ \sigma $ is the hyperelliptic involution induced by $ \lambda $. 
The function $ \lambda $ is branched  at $ x_0 =\lambda^{- 1} (0) $ and $ x_\infty =\lambda^{- 1} (\infty ) $ whilst the anti-holomorphic involution $ \rho $ covers $ \lambda\mapsto\bar\lambda^{- 1} $ and has all points with $ |\lambda|=1 $ as fixed points. The remaining data $ \lambda_0 $ is a point on the unit circle or equivalently a pair of marked points $ x_1, x_2\in X $ such that $ \lambda (x_1) =\lambda (x_2) =\lambda_0 $. Writing $ P^d $ for the set of polynomials of degree at most $ d $, we define $ \rho^\ast: P^d\rightarrow P^d $ by
\[
\rho^\ast(p (\lambda)) =\bar\lambda^d{p (\bar\lambda^{- 1})},
\]
and polynomials satisfying $ \overline {\rho^\ast p} = p $ are said to be real with respect to $ \rho $. The reality condition is equivalent to the requirement that for 
\[
p (\lambda) = x_1+ p_1\lambda + \cdots + p_d\lambda^d 
\]
we have $ p_{d-i} =\overline {p_i} $. The set of polynomials of degree at most $ d $ which are real with respect to $ \rho $ is denoted by $ P_\R^d $.

If $ (X,\lambda,\rho,\lambda_0, L) $ is the spectral data of a CMC torus then the quadruple $ (X,\lambda,\rho,\lambda_0) $ satisfies periodicity conditions. If these conditions are not satisfied, one still obtains from $ (X,\lambda,\rho, L) $ a constant mean curvature immersion of the plane into $ \R^3 $. Not all constant mean curvature immersions of the plane correspond to spectral curve data as above; those which do are said to be of {\it finite type}. To consider all constant mean curvature immersions of the plane one would need to study analytic one-dimensional varieties, it is a key result \cite{Hitchin:90, PS:89} however that all doubly-periodic such immersions correspond to {\it algebraic} curve data, that is are of finite type. Given a hyperelliptic curve with real structure and $ \lambda_0\in S^1 $, the line bundle may be chosen from a real $ g $-dimensional family, where $ g $ denotes the arithmetic genus of $ X $ and as such we see that constant mean curvature immersions of the plane of finite type come in families whose dimension is given by the spectral genus. We restrict our attention to smooth spectral curves, and we may describe such $ X = X_a $ in $ \C^2 $ by an equation of the form
\[
y^2 =\lambda a (\lambda) = (-1)^g\lambda\prod_{j = 1}^g\frac{\bar\eta_j}{\left|\eta_j\right|} (\lambda -\eta_j) (\lambda -\bar\eta_j^{- 1}),
\]
where $ a $ belongs to the space $\Hm^g\subset P^{2g}$ of polynomials of degree $ 2g $ satisfying
\begin{itemize}
\item the reality condition $ \rho^\ast a =\bar a $, 
\item $ \lambda^{- g} a (\lambda) >0 $ for all $ \lambda\in S^1 $,
\item the highest coefficient of $ a $ has absolute value $ 1 $ and
\item the roots of $ a $ are pairwise distinct, forcing $ X_a $ to be smooth.
\end{itemize}
The roots $ \eta_1,\ldots,\eta_g $ in the punctured open unit disc $ B(0,1)\setminus\{0\}\subseteq\C $ determine a unique $ a\in\Hm^g $ and hence we view $ \Hm^g $ as an open subspace of $ \C^g $, with the induced topology.

The periodicity conditions may then be phrased in terms of a pair of meromorphic differentials on the curve $ X_a $. For each $ a\in\Hm^g $, let $ \B_a $ denote the real 2-dimensional space of polynomials $ b $ of degree $ g +1 $ satisfying $b\in P^{g+1}_\R$ and such that the meromorphic differential
\[
\Theta_b := \frac{b (\lambda) d\lambda}{\lambda y}
\]
has purely imaginary periods. Each element $ b $ of $ \B_a $ is uniquely determined by $ b(0) $ up to adding a holomorphic differential, which is fixed by the condition on the periods. Therefore these elements are in one-to-one corresondence of numbers $ b(0)\in\C $.

Each family of constant mean curvature immersions of a genus one surface in $ \R^3 $ corresponds  to a pair $ (a,\lambda_0)\in\Hm^g\times S^1 $, unique up to rotations, such that there exist linearly independent $ b_1, b_2\in\B_a $ and functions $ \mu_1,\mu_2 $ on $ X_a $ satisfying
\begin{itemize}
\item The multivalued functions $ \log\mu_1 $, $ \log\mu_2 $ are holomorphic away from $ x_0 =\lambda^{-1}\{0\} $ and $ x_\infty =\lambda^{- 1}\{\infty\} $, where they have simple poles and linearly independent residues,
\item $ \Theta_{b_1} = d\log\mu_1,\quad \Theta_{b_2} = d\log\mu_2 $,
\item $ \mu_1 (\lambda_0) =\mu_2 (\lambda_0) =\pm 1 $ 
and
\item $ b_1 (\lambda_0) = 0 = b_2 (\lambda_0) $.
\end{itemize}

The point $ \lambda_0\in S^1 $ is called the Sym point.
For each $ \lambda_0\in S^1 $ the set
\[
\Pm_{\lambda_0}^g=\{a\in\Hm^g\mid X_a \text { is the spectral curve of a CMC torus in }\R^3\}
\]
is contained in the subset 
\[
\Sm^g_{\lambda_0} =\{a\in\Hm^g |\text { all } b\in \B_a \text { satisfy } b (\lambda_0) = 0 \}.
\]
We may consider the 2-dimensional real vector spaces $ \B_a $ as forming a sub-bundle $ \B^g $ of the trivial bundle $ \Hm^g\times P_\R^{g +1} $ over $ \Hm^g $. The frame bundle $ \Frame^g $ of $ \B^g $ has as elements triples $ (a, b_1, b_2) $ as above and we may consider the real subvariety of $ \Frame^g $ given by the vanishing of $ b_1 (\lambda_0) $ and $ b_2 (\lambda_0) $. 
This subvariety consists of entire fibres of the projection $ \Frame^g\rightarrow\Hm^g $ and its image under this projection is  the set $ \Sm^g_{\lambda_0} $. Hence we see that $ \Sm^g_{\lambda_0} $ is a real subvariety of $ \Hm^g $.


Using the natural action of the rotation $ \lambda\mapsto e^{i\phi}\lambda $ on the spectral curve $ X_a $ we may instead consider the spaces
\[
\Sm^g =\bigcup_{\lambda_0\in S^1}\Sm^g_{\lambda_0}
\quad\text { and }\quad
\Pm^g =\bigcup_{\lambda_0\in S^1}\Pm^g_{\lambda_0}.
\]
The set $ \Sm^g $ parametrising spectral curves such that $ \greatest (\B_a) $ has zeroes on the unit circle has the disadvantage that for $ g >2 $, it may not be a subvariety.
Thus it will frequently be more convenient to consider the subvariety
\[
\Rm^g =\{a\in\Hm^g\mid\text { all } b\in\B_a\text { have a common root}\}.
\]
The lift of $ \Rm^g $ to the frame bundle $ \Frame^g $ consists of entire fibres and exactly coincides with the vanishing set of the resultant of $ b_1, b_2 $, hence $ \Rm^g $ is indeed a subvariety. Writing $ \greatest (\B_a) $ for the greatest common divisor of $ b\in\B_a $, we may equally well characterise $ \Rm^g $ as the $ a\in\Hm^g $ satisfying $ \deg (\greatest (\B_a))\geq 1 $.
\section{Decomposition of the moduli space $ \Hm^g\setminus\Sm^g $ }
\label{section:decomposition}
We shall in this section introduce an integer invariant of $ a\in\Hm^g $, the {\it winding number}, and show that for $ g >0 $, the  sets $ V_j $ on which the winding number takes  fixed values yield  a decomposition of $ \Hm^g\setminus\Sm^g $ into $ g $ open and non-empty sets. To do this, for linearly independent $ b_1, b_2\in\B_a $ consider the function 
\[
 f =\frac{b_1}{b_2}  :\mathbb{P}^1\rightarrow\mathbb{P}^1. 
\]
This function $ f $ is defined by $ a $ up to M\"obius transformations, and hence in particular its degree $ \deg (f) $ is well-defined, independent of the choice of $ b_1, b_2 $. We wish to consider also a ``real" analogue $ \tilde f $ of $ f $, and for this purpose let $ b_0\in\B_a\otimes \C $ be a polynomial with the property that
\[
\Theta_{b_0} =\frac{b_0 (\lambda) d\lambda}{\lambda y}
\]
is holomorphic on $ X_a\setminus\{x_\infty \} $ and has a pole of order two at $ x_\infty $ with no residue. This is equivalent to requiring that $ b_0 = \alpha b_1+\beta b_2 $ for linearly independent $ b_1, b_2\in\B_a $ and $ \alpha,\beta\in\C $ determined such that $ b_0 (0) = 0 $ whilst the highest (i.e. $ \lambda^{(g +1)} $ ) coefficient of $ b_0 $ is non-zero. We make the normalisation that highest coefficient of $ b_0 $ is $-2 i $ and then 
\[
\alpha = \frac{2 i b_2 (0)}
 {b_1 (0)\overline {b_2 (0)} - b_2 (0)\overline {b_1 (0)}},
\quad\beta = \frac{-2 i b_1 (0)}{b_1 (0)\overline {b_2 (0)} - b_2 (0)\overline {b_1 (0)}}.
\]
We know that $ b\in\B_a $ is uniquely determined by $ b (0) $ and then for linearly independent $ b_1, b_2\in\B_a $ we have that the denominators of the above expressions, and the expressions themselves, are both nonzero. Defining $ b_\infty =\overline {\rho^\ast b_0} $ with $\rho^\ast$ acting on $P^{g+1}\ni b_0,b_\infty$, the differential $ \Theta_{b_\infty} $ is holomorphic on $ X_a\setminus\{x_0\} $ and has a pole of order two at $ x_0 $ with no residue.

We set
 $ 
 \tilde f = \frac{b_0}{b_\infty }
 $ 
and observe that $ \tilde f $ maps the unit circle to itself and so restricts to a map $ \tilde f : S^1\rightarrow S^1 $. We write $ \wn(\tilde f) $ for $ \deg \tilde f : S^1\rightarrow S^1 $, which we term the {\em winding number} of $ f $. 

Take  the unique $ b_1, b_2\in\B_a $ such that $ b_1 (0) = 1 $ and $ b_2 (0) = i $. With this normalisation we have $ b_0 = b_2 - i b_1 $, $ b_\infty = b_2+ i b_1 $ and
\[
 \tilde f =\frac{i b_0}{-i b_\infty}=\frac{b_1+ i b_2}{b_1 - i b_2} =\frac{f + i}{f - i}.
\]


We shall make repeated use of the polynomial $ \greatest (\B_a) $, defined as the greatest common divisor of $ b\in\B_a $, or equivalently of $ b_1, b_2 $.

\begin{lemma}\label{lemma:degree}
The degree $ \deg (f) $ and winding number $ \wn (\tilde f) $ of $ f $ satisfy
\begin{align*}
\wn(\tilde f)&\equiv\deg(f) \;\mathrm{mod}\;2&&\mbox{and}&
-\deg (f)&<\wn(\tilde{f}) \leq \deg (f).
\end{align*}
\end{lemma}

\begin{proof}
 The
 polynomials $ b_1 $ and $ b_2 $ have degree $ g +1 $, so 
\begin{equation}\label{eq:degreef}
\deg(f) = g +1 - \deg ({\greatest (\B_a)}),
\end{equation}
where we shall write $ \deg (h) $ for the number of zeros of a meromorphic function $ h:\C P^1\rightarrow\C P^1 $ and $ m_U (h) $ for the number of these occurring on the set $ U\subseteq\C P^1 $. 

Since $ b_1 (0) = 1 $ and $ b_2 (0) = i $ we have $ b_1 (0) + i b_2 (0) = 0 $ and $ b_1 (0) - ib_2 (0) = 2 $. The reality condition $b\in P^{g+1}_\R$ implies that $ b_1+ ib_2 $ is a polynomial of degree $ g +1 $ and highest coefficient $ 2 $. 
By Cauchy's argument principle,
\begin{equation}\label{equation:degreeS1}
\wn(\tilde f) = m_{B (0, 1)} (b_1+ ib_2) - m_{B (0, 1)} (b_1 - ib_2).
\end{equation}
Again using the reality of $b$, a point $ \lambda $ in the open unit disc $ B (0, 1) $ is a root of $ b_1 - ib_2 $ if and only if $ \overline\lambda^{- 1} $ is a root of $ b_1+ ib_2 $, and  all roots of $ b_1+ ib_2 $ on $ S^1 $ are common roots of $ b_1 $ and $ b_2 $. Hence
\begin{align*}
\wn(\tilde f) &= m_{\P^1} (b_1+ ib_2) - m_{S^1} (b_1+ ib_2) -2m_{\P^1 -\overline {B (0, 1)}}(b_1+ib_2)\\
&= g+1 - m_{S^1} ({\greatest (\B_a)}) -2 m_{B (0, 1)} (b_1 - ib_2)
\end{align*}
and substituting the resulting expression for $ g +1 $ into \eqref{eq:degreef} yields
\begin{align}
\deg (f) & =\wn(\tilde f) +2 m_{B (0, 1)} (b_1 - i b_2) - (\deg ({\greatest (\B_a)}) - m_{S^1} ({\greatest (\B_a)}))\nonumber\\
& =\wn(\tilde f) +2\left (m_{B (0, 1)} (b_1 - ib_2) - m_{B (0, 1)} ({\greatest (\B_a)})\right),\label{eq:nine}
\end{align}
where in the last equality we again used that the roots of $ {\greatest (\B_a)} $ are symmetric with respect to involution in the unit circle.

From  \eqref{equation:degreeS1} and \eqref{eq:nine} then 
\begin{equation}
\wn(\tilde f) +\deg (f) =2 \left(m_{B (0, 1)} (b_1+ ib_2) - m_{B (0, 1)} ({\greatest (\B_a)})\right)\label{eq:negative}
\end{equation}
and from either \eqref{eq:nine} or \eqref{eq:negative} we conclude that $ \wn(\tilde f) \equiv\deg (f)\;\mathrm{mod}\;2 $. We note either from these equations or the geometrical meaning of degree that $ -\deg (f)\leq\wn(\tilde f)\leq\deg (f) $. From \eqref{eq:negative} and the fact that $ b_1+ ib_2 $ vanishes at the origin whereas $ {\greatest (\B_a)} $ does not, we have that $ -\deg (f) <\wn(\tilde f) $.
\end{proof}

We now give a characterisation of spectral genus 0 which enables us to exclude equality in the lemma above whenever $ g >0 $.

\begin{theorem}\label{lemma:equivalent} The following three statements are equivalent.
\begin{enumerate}
\item[(i)] $ g (X_a) = 0 $.
\item[(ii)] $ \deg(f) = 1 $.
\item[(iii)] $ \deg(f) =\wn (\tilde f) $.
\end{enumerate}
Hence for $ g >0 $, the winding number of $ f $ obeys $ |\wn(\tilde f)|\le\deg(f)-2 $.
\end{theorem}
\begin{proof}{\bf(i)$ \Rightarrow $(ii):}
The unique spectral curve with $ g = 0 $ is given by $ y^2 =\lambda $. Then
\[
d\left(2 -\frac{2}{y}\right) =\frac{\lambda +1}{y}\frac{d\lambda}{\lambda}\quad\mbox{and}\quad d\left(-2 iy -\frac{2 i}{y}\right) =\frac{- i (\lambda -1)}{y}\frac{d\lambda}{\lambda}
\]
so that $ b_1(\lambda) =\lambda +1 $ and $ b_2 = - i (\lambda -1) $, giving $ f = i\left(\frac{\lambda +1}{\lambda -1}\right) $ and $ \tilde f =\lambda $. This implies $ \deg (f) = 1 =\wn(\tilde f) $.


~\\
{\bf(ii)$ \Rightarrow $(iii):} This is obvious from Lemma \ref{lemma:degree}. 
~\\
{\bf(iii)$ \Rightarrow $(i):} Assume then that $ \wn(\tilde f) =\deg(f) $. The equality of the degrees implies both that $ f (\lambda) $ is real only when $ \lambda\in S^1 $ and that the differential of $ \tilde f: S^1\rightarrow S^1 $ does not change sign for $ \lambda\in S^1 $. Writing $ x_0 $ for the unique point over $ \lambda = 0 $, let $ X^+ $ denote the non compact Riemann surface containing $ x_0 $ given by $ \lambda^{- 1}{B (0, 1)}\subseteq X_a $. The boundary $ \partial X^+ $ coincides with $ \lambda^{-1} (S^1 ) $. Now we prove (i) in seven steps.

\noindent{\bf 1. Definition of ${P: X^+\setminus\{x_0\}\rightarrow\R^2} $:}
For $ b\in\B_a $, the real part of $ \Theta_b $ is a meromorphic differential with no residues and vanishing periods. Hence we may define real-valued harmonic functions $ p_1, p_2 $ on $ X^+\setminus\{x_0\} $ by
\[
d p_1 =\Re\Theta_{b_1},\quad dp_2 =\Re\Theta_{b_2}
\]
and $ p_1, p_2 $ are unique if we further require that $ \sigma_*(p_1) = - p_1 $ and $ \sigma_*(p_1) = - p_2 $. We have then a smooth map
\[
P = (p_1, p_2): X^+\setminus\{x_0\}\rightarrow\R^2.
\]

\noindent{\bf 2. $P$ is an immersion on ${ X^+\setminus\{x_0, x_1,\ldots, x_l\}} $:} Here ${x_1,\ldots, x_l }$ are the pre-images under $\lambda$ of the roots of ${\greatest (\B_a)}$. Suppose that for some $ x\in X^+\setminus\{x_0, x_1,\ldots, x_l\} $ we have $ \alpha_1 (d p_1)_x =\alpha_2 (d p_2)_x $ for $ \alpha _1,\alpha _2\in\R $ not both zero. Locally on $ X^+\setminus\{x_0\} $ there are holomorphic functions $ q_1, q_2 $ and harmonic functions $ r_1, r_2 $ such that
\[
d q_k =\Theta_{b_k}\text { and } q_k = p_k + i r_k.
\]
Since $ d q_k $ is complex linear, 
\[
 d p_k (i v) + i d r_k (i v) = i d p_k (v) - d r_k (v) 
\]
and hence also $ \alpha _1 (d q_1)_x =\alpha _2 (d q_2)_x $ and in particular $ \alpha_1 b_1 (\lambda (x)) =\alpha _2 b_2 (\lambda (x) ) $ so since $ x $ is not a root of $ {\greatest (\B_a)} $, we have that $ \tilde f (x)\in S^1 $. But we assumed above that this occurs only for $ x\in\partial X^+ $, a contradiction and so $ P $ is indeed an immersion on $ X^+\setminus\{x_0, x_1, \ldots, x_l\} $.

\noindent{\bf 3. Definition of ${\tilde P: X^+\rightarrow\R^2}$:} We define $ \R^2\cup\{\infty\} $ to be the one-point compactification of $ \R^2 $ with its usual topology (i.e. that of the Riemann sphere) and write $ \tilde P $ for the composition of $ P $ with inversion in the unit circle. Explicitly, 
\[
 \tilde P = \dfrac {(p_1, p_2)}{p_1^2+ p_2^2} .
\]
 Note that on $ \partial X^+ $, the reality condition $b\in P^{g+1}_\R$ forces the locally defined functions $ q_k =\log\mu_k $ to be purely imaginary so $ P (\partial X^+) = 0 $ and hence $ \tilde P (\partial X^+) =\infty $. We extend $ \tilde{P} $ to a continuous map from the closure $ \bar{X}^+ $ to $ \R^2\cup\{\infty\} $.

\noindent{\bf 4. ${\tilde{P}} $ is an immersion on $ {X^+\setminus\{x_1, \ldots , x_l\}} $:} Inversion is an immersion, so we need only prove this at $ x_0 $. The normalisation $ b_1 (0) = 1, b_2 (0) = i $ forces $ g: = q_1+ i q_2 $ to be holomorphic at $ x_0 $ and we may choose the constant of integration so that $ g (0) = 0 $. Near $ x_0 $ we may define a local coordinate $ z $ for $ X $ by
\[
q_1 =\frac{1}{z},\quad q_2 =\frac{- i }{z} + g.
\]
Then
\[
p_1 =\Re\left (\frac{1}{z }\right),\quad p_2 =\Re\left ( \frac{- i }{z}\right) +\Re (g),
\]
 so
\[
p_1^2+ p_2^2 =\frac{1}{z\bar z} +2\Re\left (\frac{- i}{z}\right)\Re g + (\Re g)^2 =\frac{1}{z\bar z}\left (1+ o (z\bar z)\right),
\]
giving
\[
\frac{1}{p_1^2+ p_2^2} = z\bar z (1+ o (z\bar z))
\]
and
\[
\tilde p_1 =\frac{p_1}{p_1^2+ p_2^2} =\Re (z) (1+ o (z\bar z)),\quad
\tilde p_2 =\frac{p_2}{p_1^2+ p_2^2} =\Im(z) (1+ o (z \bar z)).
\]
Then we can explicitly differentiate by first principles and see that $ \tilde P $ is differentiable with continuous derivative. Furthermore at $ x_0 $,
\[
d \tilde p_1 =\Re (d\bar z),\quad d\tilde p_2 =\Im (d z)
\]
and hence $ \tilde P $ is an immersion at $ x_0 $.

\noindent{\bf 5. $\tilde{P}$ is a homeomorphism $X^+\setminus\{ x_1,\ldots, x_l\}\simeq\R^2\setminus\{\tilde{P}(x_1),\ldots, \tilde  P(x_l)\}$:} 
For $ m\in\Z^+ $, denote by $ W^+ (m) $ the set of points in $ \R^2\setminus\{  \tilde {P}(x_1),\ldots,  \tilde {P} (x_l)\} $ whose pre-image with respect to $ \tilde {P} $ contains at least $ m $ elements. We shall show that $ W^+ (m) $ is open and closed as a subset of $ \R^2\setminus\{  \tilde {P}(x_1),\ldots,  \tilde {P} (x_l)\} $. For any $ w =  \tilde {P} (z_1) =\cdots =  \tilde {P} (z_m)\in W^+ (m) $, since $ \tilde {P} $ is an immersion at each (distinct) $ z_j $, we may apply the $ \mathrm{C}^1 $-Inverse Function Theorem to see that a neighbourhood of $ w $ is also contained in $ W^+ (m) $ and so $ W^+ (m) $ is open. To proof closedness we take a sequence $ \{w (n) \} $ in $ W^+ (m) $, converging to $ w\in \R^2\setminus\{  \tilde {P}(x_1),\ldots,  \tilde {P} (x_l)\}$. Take for each $ j = 1, \ldots , m $ a  sequence $ z_ j (n) \in X^+\setminus\{ x_1,\ldots, x_l\} $ so that $ \tilde {P} (z_ j (n)) = w (n) $ and such that for $ i\neq j $, the sequences $ z_i (n) $ and $ z_j (n) $ are disjoint. A subsequence of each $ \{z_j (n)\} $ must converge in the compact set $ \bar{X}^+=X^+\cup\partial X^+ $ to some point $ z_j $. Hence for $ j = 1,\ldots, m $ we have $ w =  \tilde {P} (z_j) $ and since $ w\in  \R^2\setminus\{  \tilde {P}(x_1),\ldots,  \tilde {P} (x_l)\} $, then also $ z_j\in X^+\setminus\{ x_1,\ldots, x_l\} $. Hence $ \tilde {P} $ is an immersion at each $ z_j $. If $ z_i = z_j $ for $ i\neq j $ then by the Inverse Function Theorem we have $ z_i (n) = z_j (n) $ for sufficiently large $ n $, contradicting our definition of the sequences. Thus we see that $ W^+ (m) $ is closed in $ \R^2\setminus\{  \tilde {P}(x_1),\ldots,  \tilde {P} (x_l)\} $. 

The set $ W $ of points in $ \R^2\setminus\{  \tilde {P}(x_1),\ldots,  \tilde {P} (x_l)\} $ whose pre-image with respect to $ \tilde {P} $ contains exactly one element is the intersection of $ W^+ (1) $ with the complement in $ \R^2\setminus\{  \tilde {P}(x_1),\ldots,  \tilde {P} (x_l)\} $ of $ W^+ (2) $ and hence is also open and closed in $ \R^2\setminus\{  \tilde {P}(x_1),\ldots,  \tilde {P} (x_l)\} $. Moreover $ W $ is nonempty, as we shall now demonstrate. Since $ \lambda = 0 $ is not a root of $ {\greatest (\B_a)} $, applying the $ \mathrm{C}^1 $ -Implicit Function Theorem at $ x_0 $, there is a neighbourhood $ U $ of $ x_0 $ in $ \bar{X}^+ $ such that the restriction of $ \tilde {P} $ to $ U $ is an embedding onto the complement of a compact subset of $ \R^2 $. The image of $ \bar{X}^+\setminus U $ under $ \tilde {P} $ is compact, so $ W $ contains all elements in $ \R^2\setminus\{  \tilde {P}(x_1),\ldots,  \tilde {P} (x_l)\} $ outside of the union of these two compact sets in $ \R^2 $. In particular then $ W $ is non-empty so since it is both open and closed in $ \R^2\setminus\{  \tilde {P} (x_1),\ldots,  \tilde {P} (x_l)\} $ we conclude that $ W = \R^2\setminus\{  \tilde {P} (x_1),\ldots,  \tilde {P} (x_l)\} $. Since we have already shown that it is an immersion, we have that $ \tilde {P} $ is a homeomorphism from $ X^+\setminus\{ x_1,\ldots, x_l\} $ onto $ \R^2\setminus\{  \tilde {P} (x_1),\ldots,  \tilde {P}(x_l)\} $.

\noindent{\bf 6. $\tilde{P}$ is  a homeomorphism $X^+\simeq\R^2$:} If there exists $ x\in X^+ $ such that $ \tilde {P} (x) =  \tilde {P} (x_k) $ for some $ k = 1,\ldots, l $ then since $ \tilde P $ is a homeomorphism $ X^+\setminus\{x_1,\ldots, x_l\}\simeq\R^2\setminus\{  \tilde {P} (x_1),\ldots,  \tilde {P} (x_l)\} $ it must be that $ x = x_j $ for some $ j = 1,\dots, l $. Let $ U_k $ be a compact coordinate neighbourhood of $ x_k $ which does not contain any other $ x_j $. Any closed neighbourhood of $ x_k $ contained in $ U $ is compact, and so its image under the continuous map $ \tilde  P $ is compact and hence closed. Thus the local inverse of $ \tilde P $ which sends $ y =\tilde P (x_k) $ to $ x_ k $ is continuous. Similarly, the local inverse which maps $ \tilde P (x_j) =\tilde P (x_k) $ to $ x_j $ is continuous. But then since these local inverses must agree away from $ \tilde P (x_k) $ in fact they must agree also on $ \tilde P (x_k) $. Patching now yields that $ \tilde P $ is a homeomorphism $ X^+\simeq\R^2 $.

\noindent{\bf 7. $X_a$ has genus zero and $\deg f = 1$:} 
Using the real structure $ \rho $, we know that $ X^+ $ contains $ g +1 $ branch points. However we have just proven that it is simply connected, and hence $ g = 0 $.\end{proof}
For $g=0$ one clearly has $\deg(\greatest(\B_a))=0$; this theorem shows that this is true also for $g=1$. Hence $\Rm^0$ and $\Rm^1$ are empty and there are no CMC tori of spectral genus zero or one.

In proving Theorem~\ref{theorem:decomposition} below, we shall use induction on the spectral genus $ g $. The following two lemmata specify a method of deforming a given spectral curve by adding to it a small handle, and some consequences of this deformation on the function $ f $. Their proofs can be found in \cite{CS:12}, where they can be read independently from the preceding material.
\begin{lemma}\cite[Lemma~7]{CS:12} \label{lemma:polynomialperturbation}
Fix $ a\in\Hm^g $, $ b\in\B_a $, $ \alpha\in S^1 $  and a choice of $ \sqrt  {\bar\alpha} $. There exists $ \epsilon >0 $ such that for each $ t\in (-\epsilon, 0)\cup (0,\epsilon) $, 
the polynomial
\[
a_t (\lambda) = (\lambda -\alpha e^{ t})
(\bar\alpha\lambda - e^{-  t}) a (\lambda)
\]
lies in $ \Hm^{g +1} $ and for $ t\in (-\epsilon,\epsilon) $ 
the conditions
\begin{enumerate}
\item[(i)]   $ b_0 =\sqrt {\bar\alpha} (\lambda -\alpha) b $,
\item[(ii)] $ b_t (0) = -\alpha\sqrt {\bar\alpha}\, b (0) $ 
\end{enumerate}
determine a real-analytic family of polynomials $ b_t\in\B_{a_t} $. 

Note that {\rm (i)} is equivalent to $ \Theta_b =\iota_0^\ast(\Theta_{b_0}) $ where $ \iota_0: X_a\rightarrow X_{a_0} $ denotes the normalisation map $ (\lambda, y)\mapsto (\lambda,\sqrt{\bar\alpha} (\lambda - a) y) $.
\end{lemma}
\begin{lemma} \cite[Lemma~8]{CS:12}
\label{lemma:induction}
Suppose $ a\in\Hm^g\setminus\Rm^g $ 
and $ \alpha\in S^1 $ with $ df (\alpha)\neq 0 $. There exists $ \epsilon >0 $ such that for $ t\in (-\epsilon, 0)\cup (0,\epsilon) $, setting
\[
a_t (\lambda) = (\lambda -\alpha e^{ t}) (\bar\alpha\lambda - e^{-  t}) a (\lambda)\in\Hm^{g +1},
\] 
the degree of $ f_t $ satisfies
\[
\deg (f_t) =\deg (f) +1 
\]
 and hence $ a_t\in\Hm^{g +1}  \setminus\R^{g +1} $. Furthermore in addition to roots nearby those of $ df $, the differential $ df_t $ has two additional roots on $ S^1 $ in a neighbourhood of $ \lambda =\alpha $.
\end{lemma}
We claim that these two roots of $ d f_t $ are distinct. Otherwise the double root would correspond to 3 sheets of $ f_t $ coming together to form a branch point, since $ S^1 $ has no singularities. The original $ d f $ is non-vanishing at $ \alpha $ and so there is only one sheet available from the (deformation of the) original covering to join to the new sheet of $ f_t $ and so if $ \deg (f_t) =\deg (f) +1 $ we see that the new roots of $ d f_t $ are distinct.

\begin{theorem} \label{theorem:decomposition}
Define $ V_j =\{a\in\Hm^g\setminus\Sm^g \mid\wn (\tilde f) = j\} $. Then for $ g\geq 1 $, the set $ \Hm^g\setminus\Sm^g $ is the following union of non-empty, open and disjoint sets:
\[
\Hm^g\setminus\Sm^g= V_{1 - g}\cup V_{3 - g}\cup \ldots \cup V_{g -3}\cup V_{g -1}.
\]
\end{theorem}


\begin{proof}
 The continuous discrete valued function $ \wn(\tilde f) $ is  locally constant, so its level sets are open. If $ b_1 $ and $ b_2 $ have no common roots on $ S^1 $, they can only have pairs of common roots which are interchanged by $ \lambda\mapsto\bar\lambda^{-1} $. Using \eqref{eq:degreef}, this implies that $ \deg (f) \equiv (g +1)\;\mathrm{mod}\; 2 $ and then by Lemma ~\ref{lemma:degree} we have that $ \wn(\tilde f) \equiv ( g +1)\;\mathrm{mod}\; 2 $. By Lemma~\ref{lemma:degree} and Theorem~\ref{lemma:equivalent}, the values listed for $ \wn(\tilde f) $ are clearly its maximal range. We proceed to argue that these values are in fact attained.

We use induction on the genus to prove a slightly stronger statement, namely that for each $ g\geq 1 $, there exist $ a\notin\Rm^g $ for which these values of $ \wn( \tilde f) $ are attained and furthermore the corresponding $ d\tilde f $ has a simple root (and hence changes sign) somewhere on the unit circle but has no double roots on the unit circle. For $ g = 1 $, we take any $ a\neq\R^1 $ and then $ \deg (f) =  2 $ so by Lemma~\ref{lemma:degree} the only possible value for $ \wn(\tilde f) $ is zero and then the differential of $ \tilde f:S^1\to S^1 $ necessarily changes signs and hence has an odd order root. But by the Riemann-Hurwitz formula, $ d f $ has only two roots, counting multiplicity, and so they must both be simple. Since $ d f $ is complex-linear, for $ \lambda\in S^1 $ the differential is $ d\tilde f_\lambda $ and $ d f_\lambda $ have the same roots and root multiplicities. Hence any $ a\notin\R^1 $ satisfies the statement which is to be proven by induction.

Take $ a\notin\Rm^g $ satisfying the induction assumption and choose $ \alpha = e^{i\theta}\in S^1 $ at which $ d f (\alpha)\neq 0 $. For some $ t $, let $ \lambda_{1 t},\lambda_{2 t} $ be the additional roots of $ d f_t $ on the unit circle guaranteed by Lemma ~\ref{lemma:induction}. For each $ \omega\in S^1 $, there is at least one more point in $ f_t^{-1} (\omega) $ than in $ f^{- 1} (\omega) $ but the reality condition guarantees that these points occur in pairs related by involution in the unit circle. Lemma~\ref{lemma:induction} guarantees that for $ t $ sufficiently small, $ \lambda_{1 t} $ and $ \lambda_{2 t} $ are simple roots of $ d f_t $ and hence $ \tilde f_t $ changes direction at these points, as illustrated in Figure~\ref{figure:change}.

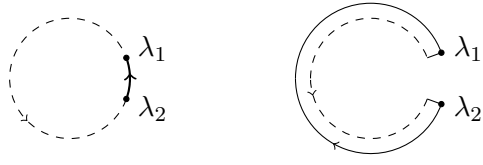
\begin{figure} [h]
\begin{tikzpicture}
\draw[dashed,
decoration={markings, mark=at position 0.625 with {\arrow{>}}},
        postaction={decorate}
](0,0) circle (0.8cm);

\draw[decoration={markings, mark=at position 0.625 with {\arrow{>}}},
        postaction={decorate}
,thick](0.7518,-0.2736) arc (-20:20: 0.8cm);

\draw[decoration={markings, mark=at position 0.625 with {\arrow{>}}},
        postaction={decorate}
](4.9397,-0.342) arc (-20:-180: 1cm);
\draw(3,0) arc (180:20: 1cm);
\draw[decoration={markings, mark=at position 0.925 with {\arrow{<}}},
        postaction={decorate}
,dashed](4.7518,-0.2736) arc (-20:-180:0.8cm);
\draw[dashed](3.2,0) arc (180:20:0.8cm);
\draw(4.9397,-0.342)--(4.7518,-0.2736);
\draw(4.9397,0.342)--(4.7518,0.2736);
\draw(1.1,0.4)node{ $ \lambda_1 $ };
\filldraw(0.7518,-0.2736)circle(1pt);
\filldraw(0.7518,0.2736)circle(1pt);
\draw(1.1,-0.4)node{ $ \lambda_2 $ };
\draw(5.3,0.4)node{ $ \lambda_1 $ };
\filldraw(4.9397,-0.342)circle(1pt);
\filldraw(4.9397,0.342)circle(1pt);
\draw(5.3,-0.4)node{ $ \lambda_2 $ };
\end{tikzpicture}
\caption {Change of direction for $ \tilde f_t $ when $ d \tilde f (\alpha ) >0 $ }
\label{figure:change}
\end{figure}
Furthermore, the roots of $ d  f_t $ aside from $ \lambda_{1 t} $ and $ \lambda_{2 t} $ are nearby the roots of $ d  f $ and we assumed that the latter were simple.  Shrinking $ \epsilon >0 $ so that for $ | t| <\epsilon $ the differential $ d\tilde f_t $ also has no multiple roots then by the reality condition, for each root $ \beta_j $ of $ d\tilde f $ there is a nearby corresponding root $ \beta_{j t} $ of $ d\tilde f_t $ (varying real-analytically in $ t $ ) and these, together with $ \lambda_{1 t} $ and $ \lambda_{2 t} $ are the only roots of $ d\tilde f_t $. Hence, writing the winding number as 
\[
\wn ( \tilde f_t) =\frac{1}{2\pi i}\int_{S^1 } d\tilde f_t,
\]
 given $ \gamma >0 $, we can choose $ \epsilon,\delta >0 $ so that for $ | t | <\epsilon $, we have
\[
\left|\frac{1}{2\pi i}\int_{e^{i (\theta +\delta)}}^{e^{i (\theta -\delta +2\pi)}} (d  \tilde f - d \tilde f_t) d\theta\right| < {\gamma} 
\]
and 
\[
\left |\frac{1 }{2\pi i}\int_{e^{i (\theta -\delta)}}^{e^{i (\theta +\delta)}} (d\tilde f - d\tilde f_t) d\theta +\sign (d \tilde f (\alpha)) 1\right | <\gamma. 
\]
Thus 
\[
\wn (\tilde f_t) =\wn (\tilde f) - \sign (d\tilde f (\alpha)) 1 
\]
and by the induction assumption both choices of sign are available. Furthermore our construction guarantees that $ d\tilde f_t $ changes sign somewhere on the unit circle and has no multiple roots. Shrinking $ \epsilon $ if necessary, we can also guarantee $ a_t\notin\R^{g +1} $.
\end{proof}

\section{Whitham Deformations and the Fermi Curve}\label{sec:whitham}
Whitham deformations will  play an important role in the remainder of this work. In this short section we introduce Whitham deformations and the Fermi curve for later reference.

 Suppose we are given a tangent vector $ (\dot{a}, \dot{b}_1,\dot{b}_2)\in T_{(a, b_1, b_2)}\Frame^g $, which infinitesimally preserves the periods of $ \Theta_{b_1},\Theta_{b_2} $. We represent the tangent vector by a path $ (a _t, b_{1 t}, b_{2 t}) $ passing through the given point when $ t = 0 $.

For $ k = 1, 2 $, the meromorphic differential forms $ \left.\frac{d}{d t }\right|_{t = 0}\Theta_{b_k} $ have vanishing periods and no residues and hence there are meromorphic functions $ \dot{q}_k $ on $ X_a $ such that 
\begin{equation}\label{equation:deformation}
 d\dot{q}_k =
\left.\frac{d}{d t }\right|_{t = 0}\Theta_{b_k} 
\end{equation}
and we may write 
\begin{equation}\label{equation:cdefinition}
\dot{q}_k = \frac{i c_k (\lambda)}{y },
\end{equation}
with $ c_k\in P^{g+1}_\R $.
Then  \eqref{equation:deformation} gives the  \emph{Whitham equation}
\[
\frac{\partial}{\partial\lambda}\frac{i c_k (\lambda)}{y} =\left.\frac{\partial}{\partial t}\frac{b_k (\lambda)}{y\lambda}\right|_{t = 0}
\]
which expands to 
\begin{align}
(2\lambda a c_1' - a c_1 -\lambda a' c_1 ) i & = 2 a\dot{b}_1 -\dot{a}b_1,\label{equation:derivative1}\\
(2\lambda ac'_2 - ac_2 -\lambda a' c_2  ) i & = 2a\dot{b}_2 -\dot{a}b_2,\label{equation:derivative2}
\end{align}
where a dot denotes the derivative with respect to $ t $, evaluated at $ t = 0 $, whilst a prime means the derivative with respect to $ \lambda $.

The compatibility of \eqref{equation:derivative1} and \eqref{equation:derivative2} gives the equation $ c_2 \text { \eqref{equation:derivative1} } - c_1\text { \eqref{equation:derivative2}} $, namely
\begin{equation}\label{equation:compatibility}
2  a\left (c_1' c_2\lambda - c_2' c_1\lambda + c_1\dot{b}_2 - c_1\dot{b}_1\right) = \dot{a}  (c_1 b_2 - c_2 b_1 ).
\end{equation}
This immediately implies that any roots of $ a $ at which $ \dot{a} $ does not vanish are necessarily roots of $ c_1 b_2 - c_2 b_1 $. Furthermore if $ \dot{a} $ vanishes at a root of a, then  \eqref{equation:derivative1} and \eqref{equation:derivative2} imply that also $ c_1 $ and $ c_2 $ vanish at this root of $ a $. Hence $ c_1b_2-c_2b_1 $ vanishes at all the roots of $ a $, and we may conclude that
\begin{equation}\label{equation:c}
c_1b_2 - c_2b_1 = Q  a
\end{equation}
with $ Q\in P^2_\R $. We may consider $ Q $ as a 1-form on $ \Frame^g $ taking values in the space $ P_\R^2 $.
Comparing \eqref{equation:compatibility} and \eqref{equation:c} we observe that
\begin{equation}\label{equation:derivative3}
2\left (c_1' c_2\lambda - c_2' c_1\lambda + c_1\dot{b}_2 - c_2\dot{b}_1\right) =\dot{a} Q.
\end{equation}
We remark also that \eqref{equation:c} can also be expressed as
\begin{equation}
\dot{q}_1dq_2 -\dot{q}_2dq_1 =\frac{Q (\lambda) d\lambda}{\lambda^2}\label{eq:Q}.
\end{equation}

We have seen how, given a tangent vector $ (\dot{a},\dot{b}_1,\dot{b}_2)\in T_{(a, b_1, b_2)}\Frame^g $ such that infinitesimal  deformations along $ (\dot{a},\dot{b}_1,\dot{b}_2) $ leave the periods of $ \Theta_{b_1} $ and $ \Theta_{b_2} $ unchanged (that is, an infinitesimal Whitham deformation), one obtains first $ c_1, c_2\in P_\R^{g +1} $ satisfying \eqref{equation:derivative1} and \eqref{equation:derivative2}  and then a  quadratic polynomial $ Q(\dot{a},\dot{b}_1,\dot{b}_2)\in P^2_\R $ satisfying \eqref{equation:c}. Conversely, one would like to reverse the process to obtain an infinitesimal Whitham deformation $ (\dot{a},\dot{b}_1,\dot{b}_2) $ from $ (a, b_1, b_2)\in\Frame^g $ and $ Q\in P^2_\R $. That is,
\begin{enumerate}
\item given $ (a, b_1, b_2)\in\Frame^g $ together with $ Q\in P^2_\R $, solve  \eqref{equation:c} for  polynomials $ c_1, c_2\in P^{g+1}_\R $,
\item given $ (a, b_1, b_2, Q, c_1, c_2) $ as above satisfying \eqref{equation:c}, solve  the system  \eqref{equation:derivative1},  \eqref{equation:derivative2} for $ (\dot{a},\dot{b}_1,\dot{b}_2)\in T_{(a, b_1, b_2)}\Frame^g $.
\end{enumerate}
 B\'{e}zout's Lemma tells us that  \eqref {equation:c} may be solved if and only if $\greatest (\B_a) $ divides $ a Q $. We are only interested in solutions $ (c_1, c_2) $ to \eqref {equation:c} which allow  \eqref {equation:derivative1} and \eqref  {equation:derivative2} to be solved and which are of degree $ g +1 $, real with respect to $\rho $, as in \eqref {equation:cdefinition}.
\begin{lemma}\label{remark:greatest}
\begin{enumerate}
\item[(i)]
Given $ (a, b_1, b_2)\in\Frame^g $ and $ Q\in P^2_\R $, there exist solutions $ (c_1, c_2, \dot{a}, \dot{b}_1, \dot{b}_2) $ of \eqref{equation:c},  \eqref{equation:derivative1},  \eqref{equation:derivative2} only if $ \greatest (\B_a) $ divides $Q$.
\item[(ii)]
Given $(a, b_1, c_1)$ and $(a, b_2, c_2)$ equations \eqref{equation:derivative1} and \eqref{equation:derivative2} may be solved if and only if $ \greatest (a, b_1) $ divides $ c_1 $ and $ \greatest (a, b_2) $ divides $ c_2 $, respectively.
\item[(iii)]\label{group c} Given $(a, b_1, b_2, c_1, c_2)$ the system \eqref{equation:derivative1}-\eqref{equation:derivative2} has a unique solution if $ \greatest (a, b_1, b_2) = 1 $.
\end{enumerate}
\end{lemma}
\begin{proof}{\bf(i):} 
Suppose the equations \eqref{equation:c},  \eqref{equation:derivative1},  \eqref{equation:derivative2} have a solution $ (c_1, c_2,\dot{a},\dot{b}_1,\dot{b}_2) $ and $ \lambda_0 $ is a root of $ \greatest (\B_a) $ of order $ n $. If it is not a root of $ a $, then by \eqref{equation:c} it must clearly be a root of $ Q $ of order at least $ n $ (and hence $ n\leq 2 $ ). If instead $ \lambda_0 $ is also a root of $ a $ of order $ k $, by the assumption of smoothness of $ X_a $ we have $ k = 1 $ and $ \lambda_0\neq 0 $. The former implies that $ a' $ does not vanish at $ \lambda_0 $ whilst the latter allows us to conclude from  \eqref{equation:derivative1} that for any solution $ c_1 $ has a root at $ \lambda_0 $ and likewise from \eqref{equation:derivative2} that $ c_2 $ has a root at $ \lambda_0 $. Thus the left hand side of \eqref{equation:c} has a root of order at least $ n +1 $ at $ \lambda_0 $, and as $ a $ vanishes there only to first order, the polynomial $ Q $ vanishes at $ \lambda_0 $ to order $ n $.

\noindent{\bf(ii):} 
Clearly $ \greatest (a, b_1) $ divides the first two terms on the left hand side of  \eqref{equation:derivative1}. Since $ \greatest (a, a') = 1 $ and $ \greatest (a,\lambda) = 1 $, the statement is immediate from B\'{e}zout's Lemma.

\noindent{\bf(iii):} 
If $ \lambda_0 $ is a common root of $ a, b_1 $ of order $ n $, then by \eqref{equation:c} and the assumption that $ a, b_1, b_2 $ have no roots in common we must have that $ \lambda_0 $ is a root of $ c_1 $ of order at least $ n $. Hence \eqref{equation:derivative1} and similarly \eqref{equation:derivative2} may be solved, for $ (\dot{a},\dot{b}_1) $ and $ (\dot{a},\dot{b}_2) $ respectively. 

The ambiguity in a mutual solution $ (\dot{a},\dot{b}_1,\dot{b}_2) $ is given by adding $ \frac{p (\lambda)}{\greatest (a, b_1, b_2)} $ for $ p $ a polynomial of degree at most $ \deg (\greatest (a, b_1, b_2)) $ satisfying the appropriate reality condition. Here this means simply that it is a real number. But the normalisation that the highest coefficient of $ a $ has absolute value one means that this real scaling is uniquely determined and hence the solution $ (\dot{a},\dot{b}_1,\dot{b}_2) $ is unique.
\end{proof}
The non-uniqueness of the B\'{e}zout coefficients $ (c_1, c_2) $ is exactly given by
\[
(\tilde c_1,\tilde c_2) = (c_1, c_2) +\frac{p(\lambda)}{\greatest(\B_a)}(b_1,b_2);
\]
  the requirement that $ \tilde c_1,\tilde c_2 $ have degree at most $ g +1 $ means that $ \deg (p)\leq\deg (\greatest (\B_a))\leq 2 $ and the requirement that they are real with respect to $ \rho $ gives $ \deg (\greatest (\B_a)) +1 $ real degrees of freedom. For some applications it will prove useful to be able to choose a solution  that suits our needs (e.g. Theorem~\ref{theorem:degreeassumption}), at other times there are natural additional assumptions that guarantee uniqueness (e.g. Theorem~\ref{theorem:Euclideanisomorphism}). 
 We may define an analytic one-dimensional variety, which we call the Fermi curve, whose normalisation is the spectral curve $ X_a $. Notice  that the locally defined map $ (q_1, q_2) $ is not a local biholomorphism from $ X_a $ to its image in $ \C^2 $ at points where $ \greatest (\B_a) $ has zeroes. The Fermi curve will have the property that the analogous locally defined map $ (q_1, q_2) $ is a local biholomorphism from the Fermi curve onto its image in $\C ^ 2 $. Usually a Fermi curve has a globally defined such map $ (q_1, q_2) $; in other words it can be realised as an analytic subvariety of the complex plane. We use here a more general notion of a Fermi curve, whose structure as a complex analytic space is only locally given by $ (q_1, q_2) $. Set 
\[
U: = X_a\setminus\{x_1, \ldots , x_m\} 
\]
where the points $ x_\beta $ are the pre-images under $ \lambda $ of the roots of $ \greatest (\B_a) $. For each $ \beta $, take a choice of local functions $ ( q_1^\beta, q_2^\beta) $ defined on a neighbourhood $ U_\beta $ of $ x_\beta $, chosen so that $ U_\beta\cap U_\gamma $ and $ V_\beta\cap V_\gamma $ are empty for $ \beta\neq\gamma $, where $ V_\beta= (q_1, q_2) (U_\beta) $. Then the Fermi curve is defined as 
\[
\left (U\cup\bigcup_\beta V_\beta\right)/\sim
\text {, where } x\sim z_\beta\text { for } x\in U, z_\beta\in V_\beta
\text { when } (q_1, q_2) (x) = z_\beta.
\]
Locally on simply connected subsets $ U_\beta $ of the spectral curve $ X_a $, there are meromorphic functions $ q_1 $ and $ q_2 $
such that $ dq_k =\Theta_{b_k} $. As in the proof of \cite[Theorem 2]{CS:12}, the differential $ \dot{q}_1dq_2 -\dot{q}_2dq_1 $ is globally defined and has the form $ \frac{Q (\lambda) d\lambda}{\lambda^2} $ where $ Q $ is a polynomial of degree two. For completeness, we give an alternative argument for this important fact.

The image of each $ U_\beta $ under the map $ (q_1, q_2) $ is the zero set of a holomorphic function $ (q_1, q_2)\mapsto R_\beta (q_1, q_2) $ on an open subset of $ \C \P^1\times\C\P^1 $. Furthermore, every point $ x\in X_a $ is contained in such an open neighbourhood $ U_\beta\subseteq X_a $, such that $ q_1:U_\beta\to V_\beta $ is a proper $ l $ -sheeted covering onto an open subset $ V_\beta\subseteq \C \P^1 $. Here $ l-1 $ is the vanishing order of $ dq_1 $ at the point $ x $. Due to \cite[Theorem~8.3]{Forster:81} there exists a unique choice of $ R_\beta $ such that $ q_2\mapsto R_\beta(q_1,q_2) $ is a polynomial of degree $ l $ with highest coefficient one, whose coefficients are meromorphic functions depending on $ q_1\in V_\beta $. For smooth families $ a_t\in\mathcal{H}^g $, these functions $ R_\beta (q_1, q_2) $ depend smoothly on $ a_t\in\mathcal{H}^g $. Hence the $ \frac{\partial R_\beta}{\partial t}  (q_1, q_2) $ are also holomorphic functions on open subsets of $ \C \P^1\times\C \P^1 $. We have that
\[
0 =\frac{d}{dt} R_\beta (q_1, q_2) =\frac{\partial R_\beta}{\partial t} (q_1, q_2) +\frac{d  q_1}{dt}\frac{\partial R_\beta }{\partial q_1}(q_1, q_2) +\frac{d  q_2}{dt}\frac{\partial R_\beta }{\partial q_2}(q_1,q_2)
\]
and for fixed $ t $,
\[
0 =\frac{\partial R_\beta}{\partial q_1} (q_1, q_2) dq_1+\frac{\partial R_\beta}{\partial q_2} (q_1, q_2) dq_2.
\]
It follows that
\begin{equation}\label{cancel}
\dot{q}_1dq_2 -\dot{q}_2dq_1
=\frac{-\frac{\partial R_\beta}{\partial t} (q_1, q_2)}{\frac{\partial R_\beta}{\partial q_1} (q_1, q_2)} dq_2=\frac{\frac{\partial R_\beta}{\partial t} (q_1, q_2)}{\frac{\partial R_\beta}{\partial q_2} (q_1, q_2)} dq_1.
\end{equation}

This equation shows that the meromorphic global differential form $ \dot{q}_1dq_2 -\dot{q}_2dq_1 $ is holomorphic away from $ x_0 $ and $ x_\infty $. Furthermore, it is invariant with respect to the hyperelliptic involution $ \sigma(\lambda,y)=(\lambda,-y) $, as $\sigma^\ast q_1=-q_1+\text{const.}$ and $\sigma^\ast q_2=-q_2+\text{const.}$ Since this form has at most third order poles at $ x_0 $ and $ x_\infty $, we conclude that it may be written as in \eqref{eq:Q}.

\section{Closure of Periodic Solutions}\label{sec:closure}

Our main result will be to give a number of equivalent, sufficient conditions for $ a\in\mathcal{H}^g $ to belong to the closure of the set $ \Pm^g (\lambda_0) $ of $ a\in\mathcal{H}^g $ such that $ X_a $ is the spectral curve of a CMC torus in $ \R^3 $. Since the Sym point lies on the unit circle, this closure is contained in the set $ \Sm^g \subseteq \Rm^g $.
 For $ g >2 $, the subvariety $ \Rm^g $ has an interesting structure as it may possess smooth points of different dimensions. In particular, we will see that  points where $ \Rm^g $ has maximal dimension, that is 

{\bf(A)} $ a\in\Rm^g $ with $ \dim_a\Rm^g = 2 g -1 $,\\ 
are contained in the closure of $ \Pm^g (\lambda_0) $. We will characterise the points where $ \Rm^g $ has maximal dimension in various ways. They are precisely the points where a small deformation can either raise or lower the winding number

{\bf(B)} $ a\in\Rm^g $ which belong to the closure of at least two different $ V_j $, $ j = 1 - g, 3 - g, \ldots , g -3, g -1 $.

Alternatively, these are points where by an arbitrarily small deformation we can obtain that the greatest common divisor of $ \B_a $ vanishes only at the Sym point

{\bf(C)} $ a\in\Rm^g $ belonging  to the closure of $ \{\tilde a\in\mathcal{H}^g \mid\deg (\greatest  (\B_{\tilde a})) = 1\} $.

As observed above, these points belong to $ \Sm^g $ and hence to some subvariety $ \Sm^g_{\lambda_0} $. Our final characterisation of points where $ \Rm^g $ has maximal dimension is

{\bf(D)} there exists $ \lambda_0\in S^1 $ such that $ a\in S^g_{\lambda_0} $ with $ \dim_a S^g_{\lambda_0} = 2g -2 $.

\vspace{10pt}

We shall demonstrate the equivalence of the four conditions listed above and then proceed to show that they force $ a $ to lie in the closure of $ \Pm^g $.
\begin{center}
{\bf(B)$ \Rightarrow$(A)}
\end{center}

\begin{lemma}\label{lemma:codimension}
Let $ R $ be a real subvariety of $ \R^n $ whose complement is a union of finitely many disjoint non-empty open sets $ \tilde V_j $. Suppose $ p\in R $ is a smooth point belonging to the closure of at least two of these sets $ \tilde V_j $. 
Then at $ p $, the subvariety $ R $ has codimension one. 
\end{lemma}
\begin{proof} We use induction on the dimension of the ambient Euclidean space $ \R^n $. For $ n = 1 $ the lemma is clear. For $ n=2 $ the lemma follows since the complement of a discrete set is a single open set. Suppose then that $ n\geq 3 $ and take 
 $ p\in R $ belonging to distinct $ \tilde V_i $ and $ \tilde V_j $. Given $ p_i\in \tilde V_i 
\setminus\{p\}, p_j\in \tilde V_j 
\setminus\{p\} $, there is at least an $ (n -3) $ -dimensional family of hyperplanes in $ \R^n $ passing through $ p, p_i, p_j $. Choosing $ p_i, p_j $ appropriately then there exists a hyperplane $ H $ in $ \R^n $ so that the intersections of $ H $ with $ \tilde V_i $ and $ \tilde V_j $ are non-empty and furthermore that there is an open punctured neighbourhood $ U $ of $ p $ in $ H \cap R $ containing only smooth points of $ R $ and such that for all $ q\in U $, the hyperplane $ H $ intersects $ R $ transversely at $ q $. Hence 
\[
\dimension_q (H  \cap R) =\dimension_q  R -1\text { for all } q\in U.
\] The complement of $ R $ in $ H $ is the union of the disjoint open sets $ H \cap \tilde V_k $ and $ p $ belongs to the closure of $ H  \cap \tilde V_i $ and $ H  \cap  \tilde V_j $, so the lemma now follows from induction on $ n $.
\end{proof}

\begin{corollary}
If $ a\in \Rm^g $ belongs to the closure of at least two different $ V_i $ and $ V_j $, then $ \dim_a\Rm^g = 2g -1 $, or equivalently, at $ a $ the variety $ \Rm^g $ has  codimension one in $ \mathcal{H}^g $.
\end{corollary}

\begin{proof} Defining
\[
\tilde V_j = V_j\cap (\mathcal{H}^g\setminus\Rm^g),
\]
since $ \mathcal{H}^g $ is open in $ \R^{2 g} $, by the previous lemma we need only consider the case when $ a $ is a singular point of $ \Rm^g $. But the smooth points of a real subvariety are open and dense  \cite[Chapter III, Theorem 1, p41]{Narasimhan:66} and  its dimension at a singular point is at least that at nearby smooth points \cite[Chapter V, Corollary following Proposition 3, page 93]{Narasimhan:66}, so since $ \Rm^g $ is a proper subvariety the corollary is proven.
\end{proof}

The implication {\bf(B)$ \Rightarrow $(A)} is then proven.

\begin{center}
{\bf(A)$ \Rightarrow $(C)}
\end{center}
\begin{lemma}\label{lemma:rotations}
The rotations $\lambda\mapsto e^{i\varphi}\lambda$ act on the frame bundle $\Frame^g$ by
\begin{align}\label{eq:rotation}
a_\varphi(\lambda)&=a(e^{i\varphi}\lambda)&
b_{1\varphi}(\lambda)&=b_1(e^{i\varphi}\lambda)&
b_{2\varphi}(\lambda)&=b_2(e^{i\varphi}\lambda).
\end{align}
This action preserves the corresponding Fermi curve. Conversely, let the 1-form $Q(\dot{a},\dot{b}_1,\dot{b}_2)$ vanish identically along a smooth section $t\mapsto(a_t,b_{1t},b_{2t})$ of the frame bundle $\Frame^g$. Then this section stays in one orbit of the action~\eqref{eq:rotation}.
\end{lemma}

\begin{proof} Away from the roots of $a$ the action~\eqref{eq:rotation} extends naturally to an action on the corresponding local maps $(q_1,q_2)$. The corresponding infinitesimal action is given by
\begin{align*}
\dot{q}_1&=\left.\frac{dq_{1\varphi}}{d\varphi}\right|_{\varphi=0}=\frac{i\lambda\Theta_{b_1}}{d\lambda}=\frac{ib_1}{y}&
\dot{q}_2&=\left.\frac{dq_{1\varphi}}{d\varphi}\right|_{\varphi=0}=\frac{i\lambda\Theta_{b_2}}{d\lambda}=\frac{ib_2}{y}.
\end{align*}
Due to~\eqref{equation:cdefinition} the corresponding $(c_1,c_2)$ are equal to $(b_1,b_2)$. Now \eqref{equation:c} implies that $Q$ vanishes. Comparing~\eqref{cancel} with \eqref{eq:Q} shows that this action preserves the corresponding Fermi curve together with all $R_\beta$.

Conversely, if $Q(\dot{a},\dot{b}_1,\dot{b}_2)$ vanishes identically along a smooth section of the frame bundle $\Frame^g$, then the Fermi curve is constant along this section. We have assumed here that $ X_a $ is smooth. Then, as in  \cite{CLP:09}, the lift of the map $ (q_1, q_2) $ to the normalisation provides a biholomorphism from the hyperelliptic curve $ X_a $ to the (analytic) normalisation of the Fermi curve. The Fermi curve does not possess an intrinsically defined degree 2 function $ \lambda $. However the hyperelliptic involution acts as $\sigma^\ast q_1=-q_1+\text{constant}$ and $\sigma^\ast q_2=-q_2+\text{constant}$. Hence $\sigma$ lifts to an involution of the Fermi curve. The normalisation of the quotient by this lift is biholomorphic to $\CP^1$. Therefore fixing the Fermi curve determines $ \lambda: X_a\rightarrow\CP^1 $ first clearly up to M\"obius transformations and then up to rotations since the points corresponding to $ \lambda = 0,\infty $ are marked. Then the presence of the anti-holomorphic involution $ \rho $ means that only rotational symmetry in $ \lambda $ remains.\end{proof}


\begin{theorem}\label{theorem:closure} 
Suppose $ \Rm^g $ has codimension one in $ \mathcal{H}^g $ at $ a\in\Rm^g $. Then $ a $ belongs to the closure of 
 $ \{\tilde a\in\mathcal{H}^g \mid\deg (\greatest (\B_{\tilde a})) = 1\} $. 
\end{theorem}
\begin{proof} 
The degree of $ b\in\B_a $ is $ g +1 $, so for $ g = 0 $ there is nothing to prove whilst for $ g = 1 $, if $ \greatest (\B_a) >1 $ then $ \deg f = 1 $, in contradiction to Lemma~\ref{lemma:equivalent}. 
Suppose then that $ g\geq 2 $. Writing $\Grass(2,\R^g) $ for the tautological bundle over the Grassmannian $ \Gr(2,\R^g) $ and $ \Frame(\Grass(2,\R^g)) $ for its frame bundle,
 consider the map
\begin{align*}
\phi:\left.\Frame^g\right|_{\Rm^g}&\rightarrow\Frame(\Grass(2,\R^g))\\
(a, b_1, b_2)  &\mapsto\Bigl(\int_{B_1}\Theta_{b_1},\ldots,\int_{B_g}\Theta_{b_1}\Bigr), \Bigl(\int_{B_1}\Theta_{b_2}, \ldots ,\int_{B_g}\Theta_{b_2}\Bigr).
\end{align*}
To see that $ \phi $ does indeed have image in the frame bundle, suppose that  for some $ a\in\Rm^g $, we have
 $ b $ in the kernel of 
\begin{align}\label{equation:injective}
\B_a &\rightarrow\R^g&
b &\mapsto \Bigl(\int_{B_1}\Theta_{b},\ldots,\int_{B_g}\Theta_{b}\Bigr).
\end{align}  
Then $ \Theta_b = d q $ for some meromorphic function $ q $ of degree 2, and we may choose $ q $ so that it is anti-symmetric with respect to the hyperelliptic involution $ \sigma $. But then if $ b \not\equiv 0 $ then $ \sigma $ has just four fixed points, contradicting the assumption $ g\geq 2 $.

Since smooth points of $ \Rm^g $ are open and dense and the dimension of a real subvariety is upper semi-continuous \cite[p41, p93]{Narasimhan:66}, a point $ a\in\Rm^g $ such that $ \Rm^g $ has dimension $ 2 g -1 $ at $ a $ is in the closure of the set of smooth points with the same property. Hence it suffices to prove the lemma under the additional assumption that $ a $ is a smooth point. The domain of $ \phi $ has dimension $ 2g +3 $ at $ a $ whilst its codomain has dimension $ 2 g $.   
We know that the rank of the differential of the smooth map $ \phi $ is upper semi-continuous, so takes its maximal value on an open and dense subset of $ \left.\Frame^g\right|_{\R_\text {smooth}} $. 
 Furthermore whenever $ d\phi_{(a , b_1, b_2)} $ has maximal rank $ 2 g $ then by the Implicit Function Theorem there is a smooth three-dimensional submanifold $ \mathcal{N} $ of $ \Frame^g $ passing through $ ( a , b_1, b_2) $ along which the periods of $ \Theta_{b _ 1},\Theta_{b_2} $ are fixed. Take a path $ (a_t, b_{1 t}, b_{2 t}) $ in $ \mathcal{N} $, passing through $ (a, b_1, b_2) $ when $ t = 0 $. Since \eqref{equation:injective} is injective, the frame bundle projection $ \pi $ of $ {\Frame^g} $ restricts to an immersion of $ \mathcal{N} $ ; let $ N =\pi (\mathcal{N}) $. 

Suppose that there is a neighbourhood $ U $ of $ a $ in $ N $ such that for all $ \tilde a\in U $, $ \deg (\greatest (\B_{\tilde a}))\geq 2 $. 
Restricting the path $ (a_t, b_{1 t}, b_{2 t}) $ to $ \mathcal{U}:=\pi^{-1} (U) $, the polynomials $ b_{1 t}, b_{2 t} $ must then have exactly  two common roots, since the degree of $ Q $ is two and by Remark~\ref{remark:greatest}, we know that $ \greatest (\B_{a_t}) $ divides $ Q $.
Hence the quadratic polynomial $ Q $ is determined up to scaling by $ a $ and so we may interpret it as a regular 1-form on $ \mathcal{U} $. Since $ \mathcal{U} $ has dimension 3, there is a 2-dimensional sub-bundle of $ T\mathcal{U} $ along which $ Q $ vanishes, and moreover by the injectivity of  \eqref{equation:injective}, this pushes forward to a 2-dimensional sub-bundle of $ T U $. We may choose two curves passing through $ a $ which have linearly independent tangent vectors at $ a $ and such that $ Q $ always vanishes along the velocity vectors to these curves.

However we know from Lemma~\ref{lemma:rotations} that there is through any point in $ a\in\mathcal{H}^g $ up to reparametrization only one curve such that $ Q $ vanishes along the velocity vector of this curve, namely the rotations~\eqref{eq:rotation}. This contradiction then forces the conclusion that $ U $ must contain points for which $ \deg (\greatest (\B_{\tilde a}))) <2 $, completing the proof of the lemma.
\end{proof}

\begin{center}
{\bf(C)$ \Rightarrow $(A), (B), (D)}
\end{center}

\begin{theorem}
\label{theorem:degreeassumption}
Suppose $ a\in\Rm^g $ satisfies $ \deg (\greatest (\B_a)) = 1 $. Then
\begin{enumerate}
\item[(i)] $ a $ is a smooth point of some $ \Sm^g_{\lambda_0} $, of codimension 2 in $ \mathcal{H}^g $ 
\item[(ii)] $ a $ is smooth in $ \Rm^g $, of codimension 1 
\item[(iii)] $ a $ belongs to the closure of two different $ V_j $ 
\end{enumerate}
\end{theorem}

\begin{proof}{\bf(i):} Given $ a $ such that $ \deg (\greatest (\B_a)) = 1 $ then by the reality condition $b\in P^{g+1}_\R$, in fact $ a\in\Sm^g_{\lambda_0} $ for some $ \lambda_0\in S^1 $. Choose $ b_1\in\B_a $ with a simple root at $ \lambda_0 $. Given any choice of $ b_2\in\B_a $ linearly independent from $ b_1 $, by replacing $ b_2 $ with an appropriate linear combination of $ b_1, b_2 $ if necessary, we may assume that $ b_2 $ has a root of order at least two at $ \lambda_0 $. In a neighbourhood $ U\in\mathcal{H}^g $ of the given $ a $ we may take a smooth choice of $ b_1, b_2 $ and hence consider the map
\begin{align*}
\eta: \;\; U &\rightarrow\R^2&
 a&\mapsto (b_1 (\lambda_0), b_2 (\lambda_0)).
\end{align*}
 We show that there exist a pair of (Whitham) deformations $ \dot{a}\in T_a\mathcal{H}^g $ such that the corresponding $ (\dot{b}_1 (\lambda_0),\dot{b}_2 (\lambda_0)) $ are linearly independent, and then (1) 
follows from the Implicit Function Theorem.


Suppose we are given a quadratic polynomial $ Q (\lambda) =\hat Q (\lambda) (\lambda -\lambda_0) $ together with $ (a, b_1, b_2) \in\Frame^g $ satisfying $ a\in\Sm_{\lambda_0} $ and $ \deg (\greatest (\B_a ) = 1 $. 
We rewrite  \eqref{equation:c} as
\begin{equation}\label{equation:chat}
c_1\hat b_2 - c_2\hat b_1 =\hat Qa,
\end{equation}
where $ b_k (\lambda) = (\lambda -\lambda_0 )\hat b_k (\lambda) $.




The condition $ \dot{b}_k(\lambda_0) = 0 $ 
is by  \eqref{equation:derivative1}, \eqref{equation:derivative2} equivalent  to the equality
\begin{equation}\label{equation:equality}
2 \lambda_0 a (\lambda_0) {c_k' (\lambda_0)}  = c_k (\lambda_0) ({\lambda_0 a' (\lambda_0) + a (\lambda_0)} ).
\end{equation}
We show first the existence of an infinitesimal deformation $ (\dot{a},\dot{b}_1,\dot{b}_2)\in T_{(a, b_1, b_2)} \Frame^g $ such that $ \dot{b}_1 (\lambda_0) $ vanishes whilst $ \dot{b}_2 (\lambda_0) $ does not.   Given $ \hat Q $, since we have assumed $ \deg (\greatest (\B_a)) = 1 $, if $ ( c_1,c_2) $ is a solution of \eqref{equation:chat} and is real with respect to $ \rho $,
then so is
\begin{equation}\label{equation:choices}
(c_1, c_2) + A (b_1, b_2) + i B\left (\frac{\lambda +\lambda_0}{\lambda -\lambda_0}\right) (b_1,{b_2} )
\end{equation}
for any $ A, B \in\R $. Note that  as $ a $ has no roots on the unit circle, $ a (\lambda_0)\neq 0 $ and so \eqref{equation:equality}  does not hold whenever 
 $ c_k (\lambda_0) = 0 $ whilst $ c_k' (\lambda_0 )\neq 0 $. Choose $ \hat Q $ with $ \hat Q (\lambda_0) = 0 $ and $ \hat Q' (\lambda_0)\neq0 $. Then  for all possible solutions $ (c_1, c_2) $ of \eqref{equation:chat}we have that $ c_2 (\lambda_0) = 0 $. Since $ a (\lambda_0)\neq 0 $, the right-hand side of  \eqref{equation:chat} vanishes to precisely first order. 

Consider first the case when $ b_2 $ vanishes to exactly order two at $ \lambda_0 $. Then by \eqref{equation:chat}, to see that there is a solution with $ c_2' (\lambda_0)\neq 0 $ it suffices to show that the freedom given by  \eqref{equation:choices} enables us to arrange that $ c_1 (\lambda_0)\neq 0 $. But clearly we may use the parameter $ B $ to arrange this, since $ \hat b_1 (\lambda)\neq 0 $. 

Now suppose that $ b_2 $ has a root of $ \lambda_0 $ of order greater than two. Then if $ c_2' (\lambda_0) = 0 $, the left-hand side of \eqref{equation:chat} vanishes at least to order two, in contradiction to the fact that its right-hand side vanishes to exactly order one. Hence in either case we see that the parameter $ B $ allows us to find a solution such that \eqref{equation:equality} does not hold for $ k = 2 $. Since $ b_2 $ vanishes to at least third-order at $ \lambda_0 $ whilst $ b_1 $ vanishes to exactly first order there, the parameter $ A $ enables us to make arbitrary changes to $ c_1' (\lambda_0) $ without changing $ c_1 (\lambda_0) $, $ c_2 (\lambda_0) $ or $ c_2' (\lambda_0) $. Hence we may choose $ A $ and $ B $ so that \eqref{equation:equality} holds for $ k = 1 $ and not for $ k = 2 $.

Since the only common root of $ b_1, b_2 $ lies on the unit circle and $ a $ can have no roots there, $ \greatest (a, b_1, b_2) = 1 $ and so by Remark~\ref{remark:greatest} for our chosen $ (\hat Q, c_1, c_2) $ there exists a (unique) infinitesimal Whitham deformation $ (\dot{a},\dot{b}_1,\dot{b}_2) $.

We now proceed to show the existence of an infinitesimal deformation such that $ \dot{b}_1 (\lambda_0) $ is non-vanishing whilst $ \dot{b}_2 (\lambda_0) = 0 $. Again it suffices to find appropriate $ (\hat Q, c_1, c_ 2) $. For any $ \hat Q $ which does not vanish at $ \lambda_0 $, the right-hand side of \eqref{equation:chat} is non-vanishing at $ \lambda_0 $ and hence for any solution $ (c_1, c_2) $ of \eqref{equation:chat} resulting from such $ \hat Q $ we have $ c_2 (\lambda_0)\neq 0 $. 

Take a preliminary choice $ \hat Q^\flat $ with $ \hat Q^\flat (\lambda_0)\neq 0 $ and corresponding $ c_k^\flat $. Then take $ \hat Q^\sharp $ with $ \hat Q^\sharp(\lambda_0) = 0 $ and consider a corresponding solution $ (c_1^\sharp, c_2^\sharp) $. Arguing as above, using the freedom given by the parameter $ B $ there is a solution so that $ c_2^\sharp (\lambda_0) = 0 $ whilst $ (c_2^\sharp)' (\lambda_0 )\neq 0 $ and the parameter $ A $ allows one to make arbitrary changes to $ (c_1^\sharp) '(\lambda_0) $ without changing $ c_1^\sharp (\lambda_0) $, $ c_2^\sharp (\lambda_0) $ or $ (c_2^\sharp)' (\lambda_0) $. The equation  \eqref{equation:chat} is linear in $ \hat Q, c_1, c_2 $. Thus with an appropriate choice of $ r, A, B \in\R $ we may take $ \hat Q =\hat Q^\flat + rQ^\sharp $ so that the resulting $ c_k = c_k^\flat + r c_k^\sharp $ fulfils \eqref{equation:equality} for $ k = 2 $ and not for $ k = 1 $.
%
%

{\bf(ii):} As $ \deg (\greatest (\B_a)) = 1 $, we may choose a neighbourhood $ U\subseteq\mathcal{H}^g $ of $ a $ such that the degree of $ \greatest (\B_{\tilde a}) $ is at most one throughout $ U $. Hence $ U \cap \Rm^g = U\cap\Sm^g $. To demonstrate that $ \Rm^g $ is smooth at $ a $ of co-dimension one, recall the polynomial $ b_0 $, defined at the beginning of section~\ref{section:decomposition}.

%
%

Using $b\in P^{g+1}_\R$, if $ \lambda_0\in S^1 $ is a root of $ b_0 $, then both $ \alpha b_1 (\lambda_0) +\beta b_2 (\lambda_0) $ and $ \overline\alpha b_1 (\lambda_0) +\overline\beta b_2 (\lambda_0) $ vanish and hence $ \lambda_0 $ is a common root of $ b_1 $ and $ b_2 $. Given a deformation $ (a (t), b_1 (t), b_2 (t)) $, define $ \lambda_0 (t) $ to be the root of $ b_0 (t) $ which when $ t = 0 $ is $ \lambda_0 $. Then applying the Implicit Function Theorem to  the map $ U\rightarrow\R $ given by $ \left|\lambda_0 \right| $, to demonstrate the smoothness of $ \Rm^g $ at $ a $ it suffices to show the existence of an infinitesimal deformation such that $ \lambda_0 $ is deformed away from $ S^1 $.

We have
\[
0 =\left.\frac{d}{d t}\right |_{t = 0}\left (b_0 (t) (\lambda_0 (t))\right) =\dot{b}_0 (\lambda_0) + b_0' (\lambda_0) \dot\lambda_0.
\]
As in the first part of the proof we may assume that $ b_1 $ has a simple root at $ \lambda_0 $ whilst $ b_2 $ has at least a double root there, so
\[
\dot\lambda_0 = -\frac{\dot{b}_0 (\lambda_0)}{b_0' (\lambda_0)} = \frac{ \alpha \dot{b}_1 (\lambda_0) - \beta\dot{b}_2 (\lambda_0)}{\alpha  b_1' (\lambda_0)}\text {, since $ b_2 ' (\lambda_0) = 0 $ }.
\]
We showed above that there exist a pair of  deformations so that the corresponding $ (\dot{b}_1 (\lambda_0),\dot{b}_2 (\lambda_0) $ are linearly independent. Using this, the above equation shows that $ \dot\lambda_0 $ can be chosen arbitrarily and so $ \Rm^g $ is smooth at $ a $.

{\bf(iii):} 
The winding number $ \wn (\tilde f) $ of $ \tilde f = -\frac{b_0}{b_\infty } $ satisfies 
\[
\wn (\tilde f) = m_{B (0, 1)} (b_\infty ) - m_{B (0, 1)} (b_0 ) = m_{B (0, 1)} (b_0) - m_{\P^1\setminus\overline {B (0, 1)}} (b_0).
\]
Hence the fact that the root $ \lambda_0 $ of $ b_0 $ can be deformed to either the inside or the outside of the unit circle shows that the point $ a $ lies in the boundary of two distinct $ V_j $.
\end{proof}

If $ a $ lies in the closure of points with $ \deg (\greatest (\B_a)) = 1 $, then by Theorem~\ref{theorem:degreeassumption} it is in the closure of smooth points of $ \Rm^g $ of codimension 1. Since $ \Rm^g $ is a proper subvariety, and dimension is upper semicontinuous, it must likewise have codimension 1 at $ a $. Similarly it belongs to some $ \Sm^g_{\lambda_0} $ and $ \Sm^g_{\lambda_0} $ has co-dimension at most two at $ a $. If $ \Sm^g_{\lambda_0} $ has codimension 1 at $ a $, then every neighbourhood of $ a $ contains a smooth point of $ \Sm^g_{\lambda_0} $ of co-dimension 1. Rotating the smooth points with respect to $ \lambda $ then shows that at them the variety $ \Rm^g $ has full dimension, a contradiction since it is proper. Thus we have shown that the condition {\bf(C)} implies the other three conditions.

\begin{center}
{\bf(D)$ \mathbf \Rightarrow $(A)}
\end{center}

\begin{lemma}
Suppose that for some $ \lambda_0\in S^1 $, the point $ a\in \Sm^g_{\lambda_0} $ satisfies $ \dim_a \Sm^g_{\lambda_0} = 2 g -2 $. Then at $ a $, the variety $ \Rm^g $ has dimension $ 2 g -1 $.
\end{lemma}
\begin{proof}
If $ a $ is a smooth point of $ \Rm^g $, then by rotating $ a $ we see that the dimension of $ \Rm^g $ at $ a $ is at least one more than that of $ \Sm_{\lambda_0} $. The smooth points of $ \Rm^g $ are dense and dimension at the singular points is at least that of nearby smooth points \cite[p 41, p 93]{Narasimhan:66} so for any $ a $ with $ \dim_a \Sm^g_{\lambda_0} = 2 g -2 $, we have that $ \dim_a\Rm^g\geq 2 g -1 $. But $ \Rm^g $ is a proper subset of $ \mathcal{H}^g $ \cite[Lemma~8 and paragraph after the proof of this Lemma]{CS:12} and so we have equality.\end{proof}
We now turn our attention to the consequences of our four equivalent conditions. On a neighbourhood of $ a\in \Sm^g $ we can choose representatives of a canonical basis $ A_1,\ldots, A_g, B_1,\ldots, B_g $ for the homology of $ X_a $ such that $ \rho_*A_i = - A_i $, $ \rho_*(B_i) = B\;\mathrm{mod}\;\langle A_1,\ldots, A_g\rangle $ and the images of the representatives under the projection $ \lambda: X_a\rightarrow\C P^1 $ depend smoothly on $ a $. Note that these reality conditions together with that of $ b\in\B _a $ forces the $ A $ -periods of the resulting differential $ \Theta_b $ to be real and hence the assumption that all periods of $ \Theta_b $ are purely imaginary means that 
\[
\int_{A_i}\Theta_b = 0\text { for } i = 1,  \ldots , g. 
\]
Furthermore we make a smooth choice of path $ \gamma $ connecting the two points of $ X_a $ whose $ \lambda $ -coordinate is $ \lambda_0 $ and such that $ \rho_*(\gamma) =\gamma\;\mathrm{mod}\; \langle A_1,\ldots , A_g \rangle $. Recall that for each $ b\in\B_a $ we write $ \Theta_b =\frac{b (\lambda)}{y} \frac{d\lambda}{\lambda} $. The map
\begin{align*}
\varphi _a: \B_a &\rightarrow\R^{g +1}&
b &\mapsto\frac{1}{2\pi i}\left (\int_{B_1}\Theta_b,\ldots,\int_{B_g}\Theta_b,\int_{\gamma}\Theta_b\right)
\end{align*}
allows to characterise $ \mathcal{H}^g (\lambda_0)\subseteq\Sm^g_{\lambda_0} $ as those $ a $ on which
\begin{align*}
\Phi : \Sm^g_{\lambda_0} &\to\Gr(2,\R^{g+1})&
a &\mapsto\varphi_a  (\B_a)
\end{align*}
has image a rational plane in $ \R^{g +1} $. 

\begin{theorem}\label{theorem:Euclideanisomorphism}
Take a smooth point $ a\in\Sm^g_{\lambda_0} $ such that $ \deg (\greatest (\B_a)) = 1 $. Then the differential $ d\Phi_a $ is injective.
\end{theorem}

\begin{proof}
Suppose $ \dot{a} =\left.\dfrac {d}{dt} a (t)\right|_{t = 0}\in\ker (d\Phi_a) $. Take infinitesimal deformations $ \dot{b}_1,\dot{b}_2 $ represented by $ b_1 (t), b_2 (t)\in\B_{a (t)} $ with 
\[
\left.\dfrac {d}{dt}\right|_{t = 0} (\varphi_{a (t)} (b_1 (t))) =\left.\dfrac {d}{dt}\right|_{t = 0} (\varphi_{a (t)} (b_2(t))) = 0.
\]

 The assumption that
\[
\left.\dfrac {d}{dt}\right|_{t=0} \int_{\gamma}\Theta_{b_1 (t)} = \left.\dfrac {d}{dt}\right|_{t=0} \int_{\gamma}\Theta_{b_2 (t)} = 0
\]
is equivalent to $ c_1 (\lambda_0) = c_2 (\lambda_0) = 0 $. Furthermore the condition $ \dot{a}\in T_a\Sm^g_{\lambda_0} $ implies that $ \dot{b}_1 (\lambda_0) = \dot{b}_2 (\lambda_0) = 0 $ and hence from  \eqref{equation:derivative1} and  \eqref{equation:derivative2} that also $ c_1' (\lambda_0) = c_2' (\lambda_0) = 0 $. From $ b_1 (\lambda_0) = b_2 (\lambda_0) = 0 $, we know that the polynomial $ Q $ has a root of at least third order at $ \lambda =\lambda_0 $, and as $ Q $ has degree two we conclude that $ Q\equiv 0 $. 

We have argued that a deformation $ (\dot{a},\dot{b}_1,\dot{b}_2) $ with $ d\Phi_a (\dot{a}) = 0 $ must have $ Q \equiv 0 $. Conversely, we now investigate what the fact that $ Q \equiv 0 $ tells us about the deformation, specifically about $ c_1 $ and $ c_2 $. We have assumed that $ \deg (\greatest (\B_a)) = 1 $, and hence  \eqref{equation:chat} has solutions $ c_1, c_2 $. The ambiguity in its solutions is precisely captured by transformations
\[
(c_1, c_2)\mapsto (c_1, c_2) + p (\lambda) (\hat b_1,\hat b_2),
\]
where in order to preserve the degree of $ c_k $, the polynomial $ p $ can have degree at most one and due to the reality conditions obeyed by $ c_k, b_k $ we have $ p (\lambda) =\alpha \lambda +\bar\alpha $. We are interested only in solutions satisfying $ c_1 (\lambda_0) = c_2 (\lambda_0) = 0 $, and since we have assumed that $ \hat b_1 = \frac{b_1}{\lambda -\lambda_0} $ and $ \hat b_2 =\frac{b_2}{\lambda -\lambda_0} $ do not have a common root this is enough to guarantee the uniqueness of $ c_k $. Then since $ b_1 $ and $ b_2 $ have only one common root counting multiplicity we must have $ c_1 \equiv c_2\equiv 0 $. Given $ (a, b_1, b_2)\in\Frame^g$ and $ (Q, c_1, c_2) $, the equations we must solve to find the corresponding deformation $ (\dot{a},\dot{b}_1,\dot{b}_2) $ are \eqref{equation:derivative1}, \eqref{equation:derivative2} and \eqref{equation:derivative3}. Here \eqref{equation:derivative3} is trivial and the remaining two equations are reduced to
\[
2 a\dot{b}_1 =\dot{a} b_1,\quad 2 a\dot{b}_2 =\dot{a} b_2.
\]
 Since we have assumed that $ b_1, b_2 $ do not have common roots other than $ \lambda_0 $, which is not a root of $ a $ we may conclude  that $ \dot{a} $ vanishes at the roots of $ a $, which forces $ \dot{a} \equiv 0 $.
\end{proof}

We come now to our main theorem.
\begin{theorem}\label{theorem:main}
For $ a\in\Rm^g $ the statements {\bf(A), (B), (C)} and {\bf(D)} are equivalent.
Moreover, if one of these equivalent conditions is satisfied than $ a $ belongs to the closure of $ \Pm^g $. This closure is contained in $ \Sm^g $.
\end{theorem}

\begin{proof} The equivalence of the four statements has already been established by the arguments of this section. To prove that such $ a $ lies in the closure of $ \Pm^g $, it suffices to consider $ a\in\Rm^g $ satisfying $ \deg (\greatest  (\B_{ a})) = 1 $. We showed in Theorem ~\ref{theorem:degreeassumption} that such an $ a $ is a smooth point of dimension $ 2g -2 $ in some $ \Sm^g_{\lambda_0} $. Then by Theorem ~\ref{theorem:Euclideanisomorphism} and the Inverse Function Theorem, on a neighbourhood of $ a $, the map $ \Phi :\Sm^g_{\lambda_0}\rightarrow\Gr(2,\R^{g+1}) $ is a local diffeomorphism onto its image, and hence $ a $ belongs to the closure of $ \mathcal{H}^g (\lambda_0) $.
\end{proof}

\section{Subvarieties of the open Grassmannian}\label{sec:grassmannian}
A subtlety which has been crucial in the above analysis is that the variety $ \Rm^g $ may contain connected components with smooth points of different dimensions. We have shown that the points of $ \Rm^g $ of highest dimension lie in the closure of $ \Pm^g $ and that these points are contained in $ \Sm^g $. It is natural to conjecture that they are all of $ \Sm^g $, but a more delicate deformation theory appears needed to prove this which we hope to undertake in future work. To better elucidate the geometric picture, in this section we analyse the corresponding subvarieties of the subset $\Gr(2,P^{g +1}_\R)^\circ $ of the Grassmannian $\Gr(2,P^{g +1}_\R)$ consisting of the 2-planes $ B $ such that the only $ b\in B $ with $ b (0) = 0 $ is the zero polynomial. If we decompose $ R^{g+1}_\R $ into the space of highest and lowest coefficients $ \simeq\R^2 $ and the remaining coefficients $ \simeq\R^g $, then the elements of $\Gr(2,P^{g+1}_\R)^\circ$ are graphs of linear maps from the former subspace $ \simeq\R^2 $ into the latter subspace $\simeq\R^g$. Furthermore these linear maps parameterize this open subset $\Gr(2,P^{g+1}_\R)^\circ $.

 Namely, in $ \Gr(2,P^{g +1}_\R)^\circ$ consider the subvarieties
\begin{align*}
R^g & : =\{B\in\Gr(2,P^{g +1}_\R)^\circ \mid \text { all } b\in B\text { have a common root.}\} \\
S^g_{\lambda_0} &: =\{B\in\Gr(2,P^{g +1}_\R)^\circ\mid\text { all } b\in B\text { satisfy } b (\lambda_0 ) = 0\}\text { for }\lambda_0\in S^1
\end{align*}
and the set 
\[
S^g: =\bigcup_{\lambda_0\in S^1} S^g_{\lambda_0}.
\]
These are easier to work with the than the spaces $ \Rm^g,\Sm^g_{\lambda_0} $ and $ \Sm^g $ and at points where the map 
\begin{align*}
\mathcal{H}^g &\rightarrow \Gr(2,P^{g +1}_\R)^\circ&
 a &\mapsto\B_a 
\end{align*}
is an immersion their properties determine those of their geometric counterparts. We see that we obtain a picture similar that of Cartan's Umbrella, with $ S^g $ providing the ``cloth'' of $ R^g $ and the remainder of $ R^g $ the ``stick''. For the analogous subsets $ \Rm^g $ we know that the ``cloth'', i.e.\ the points of highest dimension, is contained in $ \Sm^g $. 
\begin{proposition}
\begin{enumerate}
\item[(i)] If $ B\in S^g $ then $ \dim_B R^g = 2 g -1 $ and $ B $ is a smooth point of $ R^g $ if and only if $ \deg (\greatest(B)) = 1 $.
\item[(ii)] If $ B\in R^g\setminus S^g $ then $ \dim_B R^g = 2 g -2 $ and $ B $ is a smooth point of $ R^g $ if and only if $ \deg (\greatest(B)) = 2 $ (the minimum possible value for points not in $ S^g $ ).
\end{enumerate}
\end{proposition}
\begin{proof}{\bf(i):} Take $ B^0\in S^g $ with $ \deg (\greatest (B^0)) = 1 $ and write $ \lambda_0 $ for this common root.
The unique polynomial $ b_0 $ corresponding to $ B^0 $ vanishes at $ \lambda_0 $. Furthermore locally nearby $ B^0 $ the subset $ R^g $ is the set of all $ B\in\Gr(2,P^{g+1}_\R)^\circ $, with a root of the corresponding $ b_0 $ on $ \lambda\in\mathbb{S}^1 $. For any basis $ b_1,b_2 $ of $ B^0 $, the values $ \dot{b}_1(\lambda_0) $ and $ \dot{b}_2(\lambda_0) $ can be chosen independently. Therefore there always exists a direction in $ T_{B^0}\Gr(2,P^{g+1}_\R)^\circ $ with non vanishing real part of $ \dot{\lambda}_0/\lambda_0 $. Hence by the Implicit Function Theorem, $ B^0 $ is a smooth point of $ R^g $ of co-dimension 1.

The points $ B\in\Gr(2,P^{g+1}_\R)^\circ $ with $ \deg\greatest(B)=1 $ are dense in $ S^g $ and any point in $ S^g $ is contained in the closure of such smooth points. The dimension of $ S^g $ at a singular point is defined as the limit of the dimension at nearby smooth points. Therefore at every $ B\in S^g $ we have $ \dim_B R^g = 2 g -1 $.

Suppose $ B\in S^g $ has $ \deg (\greatest(B)) >1 $, and furthermore assume that the roots of $ \greatest (B) $ are distinct. Then there are $ \deg (\greatest (B)) $ distinct tangent planes of $ R^g $ passing through $ B $, which we can realise in turn by deforming $ B $ so that one common zero (or complex conjugate part of common zeros) remains common whilst the others cease to give zeros of $ \greatest (B) $. Since the set of singular points is closed and every point with $ \deg (\greatest (B)) >1 $ is in the closure of such points satisfying the additional assumption that the roots of $ \greatest (B) $ are distinct, this shows that all points where the degree of $ \greatest (B) $ is not one are singular.

\noindent{\bf(ii):} Take $ B^0\in R^g\setminus S^g $ with $ \deg (\greatest (B^0)) = 2 $. As above we may take linearly independent $ b_1^0, b_2^0\in B^0 $ with simple roots at some $ \beta^0,\overline {\beta^0}^{- 1} $ and for nearby $ B $ define corresponding $ b_1, b_2\in B $ using the normalisation $ b_k (0) = b_k^0 (0) $. Writing $ \beta _k,\overline {\beta  _k}^{- 1} $ for the corresponding pair of zeros of $ b_k $, the map
\[
B\mapsto(\Re (\beta_1 -\beta_2),\Im (\beta_1 -\beta_2)
\]
is a smooth map from a neighbourhood of $ B^0 $ into $ \R^2 $ and it has rank 2 at $ B^0 $. Then again the Implicit Function Theorem and the density of smooth points yield that points of $ R^g\setminus S^g $ with $ \deg (\greatest (B)) = 2 $ are smooth and every point of $ R^g \setminus S^g $ has dimension $ 2 g -2 $. 

Whenever $ \deg (\greatest (B) ) >2 $ the analogous argument to that given in the first part shows that $ R^g $ is singular.
\end{proof}

\begin{proposition}
Suppose
\begin{align*}
\B:\mathcal{H}^g&\rightarrow \Gr(2,P^{g +1}_\R)^\circ&
a&\mapsto\B_a
\end{align*}
is an immersion at $ a\in\Sm^g $. Then $ a $ lies in the closure of $ \Pm^g $.
\end{proposition}
\begin{proof}
The domain and range of this map have the same dimension so by the inverse function theorem if $ \B $ is an immersion at $ a $, then the spaces are locally isomorphic. In particular, $ \Rm^g $ then has dimension $ 2 g -1 $ at $ a $ so by Theorem~\ref{theorem:main}, $ a $ lies in the closure of $ \Pm^g $.
\end{proof}

\bibliographystyle{alpha}

\bibliography{harmonic}

\end{document}